\documentclass[a4paper,10pt]{amsart}

\usepackage{amsmath,amssymb,amsfonts}
\usepackage{graphicx}  
\usepackage[colorlinks=true]{hyperref}
\usepackage[boxed]{algorithm2e}

\allowdisplaybreaks

\newcommand{\gaep}{{\gamma_\eps}}
\newcommand{\om} {\Omega}
\newcommand{\omep}{{\omega_\eps}}
\newcommand{\uep}{{u_\eps}}

\newcommand{\eps}{\varepsilon}

\DeclareMathOperator{\dist}{dist}
\DeclareMathOperator{\SBV}{SBV}
\DeclareMathOperator{\divg}{div}
\newcommand{\field}[1]{\mathbb{#1}}
\newcommand{\R}{\field{R}}
\newcommand{\N}{\field{N}}
\newcommand{\Gammalim}{\Gamma\text{-}\lim}
\newcommand{\Gammaliminf}{\Gamma\text{-}\liminf}
\newcommand{\Gammalimsup}{\Gamma\text{-}\limsup}
\newcommand{\Sp}{\mathbb{S}}

\newcommand{\abs}[1]{\lvert#1\rvert}
\newcommand{\norm}[1]{\lVert#1\rVert}
\newcommand{\set}[1]{\{#1\}}

\newtheorem{theorem}{Theorem}[section]
\newtheorem*{theoremnn}{Theorem}
\newtheorem{lemma}[theorem]{Lemma}
\newtheorem{proposition}{Proposition}
\theoremstyle{definition}
\newtheorem{definition}[theorem]{Definition}

\begin{document}

\title{A variational algorithm for the detection of line segments}

\author[E.~Beretta, M.~Grasmair, M.~Muszkieta and O.~Scherzer]{}

\maketitle

\centerline{\scshape Elena Beretta}
\medskip
{\footnotesize
 \centerline{Dipartimento di Matematica}
   \centerline{Piazzale A. Moro 2}
   \centerline{00183 Roma, Italy}
}

\medskip
\centerline{\scshape Markus Grasmair}
\medskip
{\footnotesize
\centerline{Computational Science Center, University of Vienna}
\centerline{Nordbergstrasse 15, 1090 Wien, Austria}
\centerline{ and }
\centerline{Catholic University Eichst\"att--Ingolstadt}
\centerline{Ostenstrasse 26, 85072 Eichst\"att, Germany}
}
\medskip
\centerline{\scshape Monika Muszkieta}
\medskip
{\footnotesize
\centerline{Institute of Mathematics and Computer Science, 
Wroclaw University of Technology}
\centerline{
ul. Wybrzeze Wyspianskiego 27, 50-370 Wroclaw, Poland}}
\medskip
\centerline{\scshape Otmar Scherzer}
\medskip
{\footnotesize
\centerline{Computational Science Center, University of Vienna}
\centerline{Nordbergstrasse 15, 1090 Wien, Austria}
\centerline{ and }
\centerline{Johann Radon Institute for Computational and Applied Mathematics (RICAM)}
\centerline{Austrian Academy of Sciences}
\centerline{Altenbergerstrasse 69, A-4040 Linz, Austria}
}

\begin{abstract}
  In this paper we propose an algorithm for the detection of edges in
  images that is based on topological asymptotic analysis. Motivated
  from the Mumford--Shah functional, we consider a variational
  functional that penalizes oscillations outside some approximate
  edge set, which we represent as the union of a finite number of thin
  strips, the width of which is an order of magnitude smaller than
  their length. In order to find a near optimal placement of these
  strips, we compute an asymptotic expansion of the functional with
  respect to the strip size. This expansion is then employed for defining
  a (topological) gradient descent like minimization method. As
  opposed to a recently proposed method by some of the authors, which
  uses coverings with balls, the usage of strips includes some
  directional information into the method, which can be used for
  obtaining finer edges and can also result in a reduction of
  computation times.
\end{abstract}

\section{Introduction}

Detection of edges, that is, points in a digital image at which the image intensity
changes sharply is one of the most often performed steps in image
processing. Ideally, the algorithm employed for solving this problem should provide a set of connected curves that indicate the edges of objects.
In a recent paper \cite{GraMusSch11_report}, three of the authors have developed an
iterative algorithm for edge detection using the concept of
topological asymptotic analysis. The basic idea of this approach is to
cover the expected edge set with balls of radius $\eps > 0$ and use
the number of balls, multiplied with $2\eps$, as an estimate for its
length.  It was shown that under certain condition the proposed
variational model approximates 
the~Mumford--Shah functional \cite{MumSha89} in the sense of
$\Gamma$--limits, and, therefore, this algorithm may be considered as
a computational method for the approximate minimization of the
Mumford--Shah functional.
A criterion for the optimal positioning of balls covering the edge set is provided by the leading term of a topological asymptotic expansion of the approximating functional. The (iterative) implementation of the algorithm selects edges successively according to 
certain rules. In a follow up paper \cite{DonGraKanSch12_report}, it
was shown that this approach is useful for scale detection of edges.

In this paper, we consider again the problem of edge detection in the
framework of topological asymptotic analysis. As opposed to the
previous work, however, we consider now covering the edge set with
\emph{line segments} rather than with balls. 
There are several reasons:
First, edges should rather be seen as a union of small line segments
than as accumulations of points. Second, numerically, the resulting
algorithm is expected to be faster, as in each
iteration step a whole set of edge points (the segment) is detected and not a single 
point only. 
We admit here that there is still a conceptual misfit between the continuous formulation and the discrete setting. 
Theoretically, by our analysis, only edge segments can be detected that display a certain distance from the previously detected ones 
(this will be reflected in the constant $\delta_0$ below). We believe that this technical problem can in fact be solved, but it seems that this requires a much more sophisticated analysis of the topological expansion. In fact, for practical realizations, it is not a severe restriction, since the distance can, theoretically, be chosen arbitrarily small, in particular
below half of the pixel size, in which case the union of line segments
appears closed. However, compared to \cite{GraMusSch11_report}
the effect is less pronounced, because the covering line segments are relatively larger than the balls.

The novelty of this paper is an algorithm for edge detection based on
the asymptotic analysis for topological derivatives with respect to
line segments.
We note that the topological asymptotic expansion in
\cite{GraMusSch11_report} has been derived in
the framework of potential theory \cite{VogVol00}. However, in the
present case, due to the more complex geometry of the inhomogeneities
and the impossibility of introducing a uniform scaling, this approach
fails. To avoid these difficulties, in this paper we build up on a geometry independent approach of Capdeboscq \& Vogelius \cite{CapVog03a,CapVog06}. To outline our method,  we have to introduce some notation first.

Let $\Omega$ be an open and bounded subset of $\mathbb{R}^2$. We
assume that a given image $f\colon \Omega\rightarrow \mathbb{R}$ is a
bounded function that assigns to each point $x\in\Omega$ some gray
value $f(x) \in \R$.

\begin{definition}
\label{def:notation}
We denote by
\begin{equation}
\label{eq:sigma_eps}
\sigma_\eps(y,\tau):=\{x\in\mathbb{R}^2 \,:\, x = y+\rho\tau\,, \  -\eps\leq\rho\leq\eps\}
\end{equation}
a line segment of length $2\eps>0$ centered at $y\in \R^2$ and with the unit tangent vector $\tau \in \Sp^1$. 
Moreover, we define a thin strip around $\sigma_\eps(y,\tau)$ as
\begin{equation}
\label{eq:omega_eps}
\omega_\eps(y,\tau) := \{x \in \R^2 \,:\, \dist(x,\sigma_\eps(y,\tau)) \le \eps^2\}\;.
\end{equation}
If $K\subset \Omega$ is a closed subset and $0 < \kappa < 1$, we
define the function $v_{K}\colon\Omega\to\R$ by
\begin{equation}\label{eq:v_K}
v_K(x) := \begin{cases}
\kappa &\text{ if } x \in K,\\
1 & \text{ else.}
\end{cases}
\end{equation}
In particular, we will apply this notion if $K$ is the union of strips
$\omega_\eps(y,\tau)$.
Finally, for every $v \in L^2(\Omega)$ we define 
\begin{equation}
\label{eq:m_eps}
m_{\eps}(v) := \inf\bigl\{\abs{S} : S \subset \R^2\times
\mathbb{S}^1,\ v = v_{K} \text{ with } K = \bigcup_{(y,\tau)\in
  S} \omega_\eps(y,\tau)\bigr\}\;.
\end{equation}
Here we set $m_{\eps}(v) := +\infty$, if $v\neq v_{K}$
for every finite subset $S \subset \R^2 \times \Sp^1$ with $K =
\bigcup_{(y,\tau)} \omega_\eps(y,\tau)$.
\end{definition}
With this notation at hand, we introduce the functional 
\begin{equation}
\label{Jeps}
\mathcal{J}_{\eps}(u,v) := \frac{1}{2}\int_\Omega (u-f)^2\,dx + 
\frac{\alpha}{2}\int_{\Omega} v\abs{\nabla u}^2\,dx 
+ 2\beta\eps\,m_{\eps}(v)\;,
\end{equation}
which is to be minimized over all functions $u\in H^1(\Omega)$ and 
$v \in L^\infty(\Omega)$. Here $\alpha$ and $\beta$ are some positive parameters.

For the approximate numerical minimization of $\mathcal{J}_{\eps}$ we
use a topological asymptotic expansion. Defining 
\[ 
\label{eq:g}
\mathcal{J}(u,v)=\frac{1}{2} \int_\Omega (u-f)^2\,dx + \frac{\alpha}{2} \int_\Omega v \abs{\nabla u}^2\,dx\,,
\] 
we see that for general $\omega_\eps(y,\tau) \cap K = \emptyset$ we have
\[
{\mathcal J}_{\eps}(u,v_{K \cup \omega_\eps(y,\tau)}) - {\mathcal J}_{\eps}(\hat{u},v_K) = 
\mathcal{J}(u,v_{K \cup \omega_\eps(y,\tau)})-\mathcal{J}(\hat{u},v_K)+2 \beta \eps\;.
\] 
Thus the largest decrease of ${\mathcal J}_{\eps}$ with respect to a strip $\omega_\eps(y,\tau)$
can as well be found by optimizing $\mathcal{J}$ with respect to $y$ and $\tau$.
Let now $K$ be some subset of $\Omega$; in particular, it can be the
union of a finite number of thin strips. 
Now assume that we cut out a small strip $\omega_\eps(y,\tau)$ of $\Omega\setminus K$
and denote by $v_K$ and $v_{K\cup\omega_\eps} := v_{K \cup \omega_\eps(y,\tau)}$ the corresponding edge indicators.
Denote moreover by $u_K$ and $u_{K\cup\omega_\eps}$ the minimizers of the functionals
$\mathcal{J}(\cdot,v_K)$ and $\mathcal{J}(\cdot,v_{K \cup \omega_\eps})$, respectively.
Our main result in Section~\ref{sec_functexpan} is the derivation of an expansion of the form
\begin{equation}
\label{eq:approx}
{\mathcal J}(u_{K\cup\omega_\eps},v_{K \cup \omega_\eps}) - {\mathcal J}(u_K,v_K) 
\approx -2\alpha(1-\kappa)\eps^3\mathcal M\nabla u_K(y)\cdot\nabla u_K(y)\,,
\end{equation}
where $\mathcal M=\frac{1}{\kappa}n\otimes n+\tau\otimes\tau$ and $n$, $\tau$ are the unit normal and tangent vectors to the segment 
$\sigma_{\eps}$, and the intersection of $K$ and $\omega_\eps(y,\tau)$ is empty.
The above difference \eqref{eq:approx} is asymptotically valid whenever a strip is removed from the potential edge set
and can be used for finding the 
points of $\Omega$ where we can expect the largest decrease of ${\mathcal J}_{\eps}$ by removing small strips.

\section{Asymptotic expansion\label{sec_functexpan}}

We assume that $\om\subset\mathbb{R}^{2}$ is an open bounded smooth domain and $f:\Omega\rightarrow \mathbb{R}$ is a given function in $L^{\infty}(\Omega)$. We define the functional 
\begin{equation}\label{G}
\mathcal J(u,v):= \frac{1}{2}\int_{\Omega}(u-f)^2\,dx
+\frac{\alpha}{2}\int_{\Omega} v |\nabla u|^2 \,dx\;,
\end{equation}
for $u\in H^1(\om)$ and $v \in L^\infty(\Omega)$, and the parameter  $\alpha>0$.

Now assume that $K$ is a fixed open subset of $\Omega$ and define the function $v\colon\om\rightarrow\R$ by
\begin{equation}\
v(x)=\begin{cases}
\kappa& x\in K,\\
1& x\in\Omega\backslash \overline K,\end{cases}
\end{equation}
with $0<\kappa<1$. Using standard results of calculus of variations, one can show that the unique minimizer $u\in H^1(\Omega)$ of $\mathcal{J}(\cdot,v)$ is the unique weak solution to the boundary value problem
\begin{equation}\label{bv}
\left\{
  \begin{aligned}
  u-\alpha\divg(v\nabla u)       &= f &&\text{in }\Omega,\\
  \dfrac{\partial u}{\partial \nu} &= 0 && \text{on }\partial\Omega\,.
\end{aligned}\right.
\end{equation}
In the remainder of this section, we will derive a variation of the functional $\mathcal{J}$ with respect to perturbation of the function $v$ obtained by adding a small strip to the set $K$.
More precisely, let us denote by $L_0$ a compact subset of $\om\setminus\overline{K}$ such that
\[
\textrm{dist }(L_0,\partial\om\cup \overline K)\geq \delta_0>0.
\]
Let $y\in \text{int}(L_0)$ and $\tau\in \mathbb{S}^1$. We 
choose $\eps>0$ small enough so that
the thin strip~$\omega_{\eps}(y,\tau)$ defined as in (\ref{eq:omega_eps}) is contained in $L_0$.
From now on, in order to simplify the notation, we set
\[
\omega_{\eps}:=\omega_{\eps}(y,\tau)\textrm{ and } \sigma_{\eps}:= \sigma_{\eps}(y,\tau).
\]
We define the function $v_\eps\colon\om\rightarrow\R$ by
\begin{equation}\
v_{\eps}(x)=\begin{cases}
\kappa& x\in K\cup \omega_{\eps},\\
1& x\in\Omega\backslash (\overline K\cup \overline{\omega_{\eps}}).\end{cases}
\end{equation}
Similarly as above, we note that the unique minimizer $\uep \in H^1(\Omega)$ of $\mathcal{J}(\cdot,v_{\eps})$ is the unique weak solution to the boundary value problem
\begin{equation}\label{bvp}
\left\{ \begin{aligned}
  \uep-\alpha\divg (v_{\eps}\nabla \uep) &= f &&\text{in }\Omega,\\
  \dfrac{\partial \uep}{\partial \nu} &= 0 && \text{on }\partial\Omega\,.
\end{aligned}\right.
\end{equation}

Our goal is to establish an expansion for $\mathcal J(\uep,v_{\eps})-\mathcal J(u,v)$ in powers of $\eps$ as $\eps\rightarrow 0$. We will prove the following theorem:

\begin{theorem}
\label{thm:Gasymptotic}
We have
\[
\lim_{\eps\to 0} \max_{y\in L_0} \,\dfrac{1}{4\eps^3}
\left|\mathcal J(\uep, v_{\eps})- \mathcal J(u,v)
-\frac{\alpha(\kappa-1)}{2}4\eps^3\mathcal M\nabla u(y)\cdot\nabla u(y)\right|  = 0\,,
\]
where $\mathcal M=\frac{1}{\kappa}n\otimes n+\tau\otimes\tau$ and $n$, $\tau$ are the unit normal and unit tangent vectors to the segment $\sigma_{\eps}$. 
\end{theorem}

In order to prove Theorem \ref{thm:Gasymptotic} we will follow the approach of Capdeboscq \& Vogelius \cite{CapVog03a,CapVog06}. 
We will need the set
\[
\hat L_0:= L_0+\overline{B_{\delta_0/2}(0)},
\]
which is constructed in such a way that it satisfies
\[
L_0\subset \hat L_0\subset\om\setminus\bar{K} \textrm{ and } \textrm{dist}(\hat L_0,\partial\om\cup \overline K)\geq \delta_0/2,
\]
and several auxiliary lemmas.

\begin{lemma} \label{Loc}
The following identity holds:
\begin{equation}\label{locint}
\mathcal J(\uep, v_{\eps})- \mathcal J(u,v)=\frac{\alpha(\kappa-1)}{2}\int_{\omega_{\eps}}\nabla \uep \cdot \nabla u \,dx.
\end{equation}

\end{lemma}

\begin{proof}[Proof of Lemma \ref{Loc}]
Since $\uep$ and $u$ are weak solutions of \eqref{bv} and  \eqref{bvp} they satisfy, respectively,
\begin{align}\label{bvpw}
\int_{\om}(\uep-f)\phi+\alpha v_{\eps}\nabla \uep\cdot \nabla\phi\,
dx&=0 &&\text{for all } \phi\in H^1(\om),\\
\label{bvw}
\int_{\om}(u-f)\phi+\alpha v\nabla u\cdot \nabla\phi\,
dx&=0 &&\text{for all } \phi\in H^1(\om).
\end{align}
Setting $\phi=u$ in (\ref{bvpw}) and $\phi=\uep$ in (\ref{bvw}) and subtracting (\ref{bvw}) from (\ref{bvpw}) we get
\begin{equation}\label{loc1}
\int_{\om}f(\uep-u)\, dx=\alpha(1-\kappa)\int_{\omega_{\eps}}\nabla\uep\cdot\nabla u\,dx
\end{equation}
On the other hand, inserting $\phi=\uep$ in (\ref{bvpw}) and $\phi=u$ in (\ref{bvw}), we obtain, respectively,
\begin{equation}\label{loc2}
\int_{\om}\uep(\uep-f)+\alpha v_{\eps}|\nabla \uep|^2\, dx=0\,,
\end{equation}
\begin{equation}\label{loc3}
\int_{\om}u(u-f)+\alpha v|\nabla u|^2\, dx=0\,.
\end{equation}
Now
\[
2(\mathcal J(\uep, v_{\eps})-\mathcal J(u,v))=\int_{\om}(\uep-f)^2+\alpha v_{\eps}|\nabla \uep|^2\, dx-\int_{\om}(u-f)^2-\alpha v|\nabla u|^2\, dx
\]
and, by (\ref{loc2}) and (\ref{loc3}), we have
\[
\begin{aligned}
2(\mathcal J(\uep, v_{\eps})-\mathcal J(u,v))&=\int_{\om}(\uep-f)^2-(\uep-f)\uep-(u-f)^2+(u-f)u\, dx\\
&=-\int_{\om}f(\uep-u)\, dx\,.
\end{aligned}
\]
Finally, recalling (\ref{loc1}), the claim follows.
\end{proof}

\begin{lemma} \label{Reg}
The function $u$ satisfies
 \begin{equation}\label{reg1}
u\in C^{1,\lambda}(L_0)
\end{equation}
for every $0 < \lambda < 1$.
Moreover there exists a constant $C=C(\delta_0,\om)$ such that 
\begin{equation}\label{reg2}
\|\nabla u\|_{L^{\infty}(L_0)}\leq C(\|f\|_{H^{-1}(\om)}+\|f\|_{L^{\infty}(\om)})
\end{equation}
\end{lemma}

\begin{proof}[Proof of Lemma \ref{Reg}]
First we observe that, in $\om\backslash \overline K$, $u$ solves 
\[
u-\alpha\Delta u=f.
\]
Now, let $\tilde{x}\in L_0$ and let $\varphi\in C_0^{\infty}(\Omega)$
be a function with a compact support in $L_0$ such that
$\varphi(x)\equiv 1$ in a neighborhood $U$ of $\tilde{x}$. Since
$u-f$ is bounded, we get that $w = (u-f)\varphi \in L^p(L_0)$ for all
$p<+\infty$, and therefore, $\Delta^{-1} w \in W^{2,p}(L_0)$ for all
$p<+\infty$.  In particular we have that $u \in W^{2,p}(U)$ for all
$p<+\infty$. From the Sobolev imbedding theorem and
since $\tilde{x} \in L_0$ is arbitrary, we conclude that $u\in C^{1,\lambda}(L_0)$ with $\lambda\in(0,1-2/p)$ for $2<p<+\infty$.
Moreover from \cite{LadUra68b}  we have
\[
\begin{aligned}
\|u\|_{W^{2,p}( L_0)}&\leq C(\|u\|_{L^{p}(\om)}+\|f\|_{L^{p}(\om)})\\
&\leq C(\|f\|_{H^{-1}(\om)}+\|f\|_{L^{\infty}(\om)})
\end{aligned}
\]
for any $p>2$ and where $C$ depends on $\delta_0,\om$. Finally the Sobolev imbedding theorem  implies (\ref{reg2}).
\end{proof}

We now derive energy estimates for $\uep-u$.

\begin{lemma} \label{EnEst}
There exists a constant $C=C(\kappa,\delta_0)$ such that 
\begin{equation}\label{energy1}
\|u_{\eps}-u\|_{H^1(\om)}\leq C(\|f\|_{H^{-1}(\om)}+\|f\|_{L^{\infty}(\om)})|\omega_{\eps}|^{\frac{1}{2}}
\end{equation}
and 
\begin{equation}\label{energy2}
\|u_{\eps}-u\|_{L^2(\om)}\leq C(\|f\|_{H^{-1}(\om)}+\|f\|_{L^{\infty}(\om)})|\omega_{\eps}|^{\frac{1}{2}+\eta}
\end{equation}
 for some $\eta>0$.

\end{lemma}
\begin{proof}[Proof of Lemma \ref{EnEst}]
Subtracting (\ref{bvw}) from (\ref{bvpw}) we get
\[
\int_{\om}(\uep-u)\phi+\alpha v_{\eps}\nabla (\uep-u)\cdot \nabla\phi\, dx=\alpha(1-\kappa)\int_{\omega_{\eps}}\nabla u\cdot\nabla\phi\, dx \quad \forall \phi\in H^1(\om)\,.
\]
Setting $\phi=\uep-u$, inserting it in the last equality and applying Schwarz' inequality, we get
\[
\alpha\kappa \|\uep-u\|^2_{H^1(\om)}\leq\alpha (1-\kappa)\|\nabla u\|_{L^2(\omega_{\eps})}\|\uep-u\|_{H^1(\om)}.
\]
Hence
\[
\|\uep-u\|_{H^1(\om)}\leq \frac{1-\kappa}{\kappa}\|\nabla u\|_{L^2(\omega_{\eps})}
\]
and by Schwarz' inequality and the regularity estimates proved in Lemma \ref{Reg} for $u$ we derive
\[
\|\uep-u\|_{H^1(\om)}\leq C(\|f\|_{H^{-1}(\om)}+\|f\|_{L^{\infty}(\om)})|\omega_{\eps}|^{\frac{1}{2}}.
\]
To prove (\ref{energy2}) we subtract (\ref{bvpw}) from (\ref{bvw}) getting
\begin{equation}\label{en1}
\int_{\om}
(u-\uep)w+\alpha v\nabla(u-\uep)\cdot\nabla w\, dx=\alpha(\kappa-1)\int_{\omep}\nabla\uep\cdot\nabla w\, dx\quad\forall w\in H^1(\om).
\end{equation}
Let $w\in H^1(\om)$ be the solution to
\begin{equation}\label{auxp}
\left\{ \begin{aligned}
  w-\alpha\divg (v\nabla w) &= u-\uep &&\text{in }\Omega,\\
  \dfrac{\partial w}{\partial \nu} &= 0 && \text{on }\partial\Omega\,.
\end{aligned}\right.
\end{equation}

Since $w-\alpha\Delta w=u-\uep$ in $\hat L_0$, by interior regularity
results (cf.~\cite[Thm.~8.8]{GilTru01}) we have 
\[
\|w\|_{H^2(\hat L_0)}\leq C(\|u-\uep\|_{L^2(\om)}+\|w\|_{H^1(\om)}).
\]
Moreover since
\[
\|w\|_{H^1(\om)}\leq C \|u-\uep\|_{L^2(\om)}
\]
we have that 
\[
\|w\|_{H^2(\hat L_0)}\leq C\|u-\uep\|_{L^2(\om)}.
\]
By the Sobolev imbedding theorem, the last inequality implies that $\nabla w\in L^p(\hat L_0)$ for any $p\in (1,+\infty)$ and 
\[
\|\nabla w\|_{L^p(\hat L_0)}\leq C\|u-\uep\|_{L^2(\om)}.
\]
Let us now choose $q\in (1,2)$ and $p$ so that
$\frac{1}{p}+\frac{1}{q}=1$. Then, combining the variational
formulation of the problem (\ref{auxp}) with (\ref{en1}) and applying
H\"older's inequality we get
\begin{equation}
\begin{aligned}
\int_{\om}(u-\uep)^2\, dx&=\alpha(\kappa-1)\int_{\omep}\nabla\uep\cdot\nabla w\, dx\\
&\leq C\|\nabla\uep\|_{L^q(\omep)}\|\nabla w\|_{L^p(\omep)}\\
&\leq C\|\nabla\uep\|_{L^q(\omep)}\|u-\uep\|_{L^2(\om)}
\end{aligned}
\end{equation}
and since $\frac{1}{q}>\frac{1}{2}$ the claim follows.
\end{proof}
We recall here a general, geometry independent, result due to Capdeboscq \& Vogelius (cf. \cite{CapVog03a,CapVog06}).  
Let us indicate with $V^j:=x_j-\frac{1}{|\partial\om|}\int_{\partial\om}x_j\, d\sigma$, $j=1,2$, the so called corrector terms. Let 
\begin{equation}\
\gamma_{\eps}(x)=\begin{cases}
\kappa& x\in \omega_{\eps},\\
1& x\in\Omega\backslash \overline{\omega_{\eps}},\end{cases}
\end{equation}
and let $V_{\eps}^j$, $j=1,2$, be the solutions to

\begin{equation}
\left\{ \begin{array}{rcll}
  \divg (\gaep \nabla V_\eps^{j}) &=& 0 &\text{in }\Omega\,,\\[0.1cm]
  \dfrac{\partial V_{\eps}^{j}}{\partial \nu} &=& \nu_{j} & \text{on }\partial\Omega\,,\\[0.2cm]
\int_{\partial\om}V_{\eps}^{j}\, d\sigma &=& 0  \,.
\end{array}\right.
\end{equation}

Proceeding with similar arguments as in Lemma \ref{EnEst} one easily sees that
there exists a constant $C=C(\kappa,\delta_0)$ such that 
\begin{equation}\label{energy3}
\|V^j_{\eps}-V^j\|_{H^1(\om)}\leq C|\omega_{\eps}|^{\frac{1}{2}}
\end{equation}
and 
\begin{equation}\label{energy4}
\|V^j_{\eps}-V^j\|_{L^2(\om)}\leq C|\omega_{\eps}|^{\frac{1}{2}+\eta}
\end{equation}
for $ j=1,2$ and for some $\eta>0$.
Observe now that, as $\eps\rightarrow 0$,
\begin{equation}\label{mu}
\left|\omep\right|^{-1}1_{\omep}(\cdot)\textrm{ converges in the sense of measure to }
\mu\end{equation}
and the  Borel measure $\mu$ is concentrated on $L_0$. In fact, due to
the form of the set $\omep$,  it is immediate to see that
$\mu=\delta_y$, 
where $\delta_y$ denotes the Dirac measure concentrated at $y$. 
Using (\ref{energy3})  and from the analysis in \cite{CapVog03a} it
follows that, possibly up to the extraction of a subsequence, 
\begin{equation}\label{M}
\left|\omep\right|^{-1}1_{\omep}\frac{\partial V^j_{\eps}}{\partial x_i}(\cdot)\textrm{ converges in the sense of measure to }
\mathcal{M}_{ij}\textrm{ when }\eps\rightarrow 0\,,\end{equation}
where $\mathcal{M}_{ij}$ is a Borel measure with support in
$L_0$.
Again, the fact that the set $\omega_\eps$ shrinks to the point $y$
implies that the measure $\mathcal{M}_{ij}$ is simply a multiple of
$\delta_y$.
 Hence, identifying $\mathcal{M}_{ij}$ with
 $\mathcal{M}_{ij}\delta_y$, we have
\begin{equation}\label{defM1}
\mathcal{M}_{ij} \phi(y)
=\lim_{\eps\rightarrow0}\dfrac{1}{|\omep|}\int_{\omep}\dfrac{\partial V_{\eps}^{j}}{\partial x_{i}}(x)\phi(x)\, dx
\end{equation}
for every smooth function $\phi$.
Following \cite{CapVog03a} it is possible to show the following result:

\begin{lemma} \label{nu}
Let $\eps\rightarrow 0$ be such that (\ref{mu}) and (\ref{M}) hold. Then  
\begin{equation}\label{nu1}
\lim_{\eps\rightarrow0} \left|\omep\right|^{-1}\int_{\om}1_{\omep}\frac{\partial
  u_{\eps}}{\partial x_j}\phi\, dx =
\frac{\partial u}{\partial x_i}(y) \mathcal{M}_{ij} \phi(y)
\quad \forall \phi\in C_0^1( \hat L_0)
\end{equation}
for any $j=1,2$.
\end{lemma}
\begin{proof}[Proof of Lemma \ref{nu}]
From the energy estimates we have that, possibly extracting a subsequence that we do not relabel, 
\begin{equation}\label{nu2}
\left|\omep\right|^{-1}1_{\omep}\frac{\partial u_{\eps}}{\partial x_j}(\cdot)\textrm{ converges in the sense of measure to }\bar{\nu}_j,
\end{equation}
that is,
\[
\lim_{\eps\rightarrow 0}|\omep|^{-1}\int_{\om}\frac{\partial u_{\eps}}{\partial x_j}\phi\, dx=\int_{\om}\phi\, d\bar{\nu}_j
\]
for all continuous function $\phi$ in $\om$. 
In order to prove (\ref{nu1}) we will prove the following relation for $\eps\rightarrow 0$
\begin{equation}\label{nuM}
|\omep|^{-1}\int_{\omep}\nabla u\cdot\nabla V^j_{\eps}\phi\, dx=|\omep|^{-1}\int_{\omep}\nabla \uep\cdot\nabla V^j\phi\, dx +o(1),\quad \forall \phi\in C^1_0(\hat L_0).
\end{equation}
Once (\ref{nuM}) is proved, passing to the limit as $\eps\rightarrow 0$, we get 
\[
\int_{\om}\phi \, d\bar{\nu}_j=
\frac{\partial u}{\partial x_i}(y) \mathcal{M}_{ij} \phi(y)
\quad\forall \phi\in C^1_0(\hat L_0),
\]
from which (\ref{nu1}) follows.

 Hence let us prove (\ref{nuM}).
Let us notice that, if $\phi\in C^1_0(\hat L_0)$, 
then $\phi v_{\eps}=\phi\gamma_{\eps}$ and $\phi v=\phi \gamma_0=\phi$, and we have
\begin{equation}\label{nu3}
\int_{\om}(\uep-f)\phi+\alpha\gamma_{\eps}\nabla\uep\cdot\nabla\phi\, dx=\int_{\om}(u-f)\phi+\alpha
\nabla u\cdot\nabla\phi\, dx
\end{equation}
and
\begin{equation}\label{nu4}
\int_{\om}\gamma_{\eps}\nabla V^j_{\eps}\cdot\nabla\phi\,
dx=\int_{\om}\nabla V^j\cdot\nabla\phi\, dx.
\end{equation}
Using (\ref{nu3}) and (\ref{nu4}) and after some algebraic manipulations we get
\begin{align*}
\lefteqn{\alpha\int_{\om}(1-\gamma_{\eps})\nabla u\cdot\nabla V^j_{\eps}\phi\, dx-\alpha\int_{\om}(1-\gamma_{\eps})\nabla\uep\cdot\nabla V^j\phi\, dx}\\
&=\alpha\int_{\om}\nabla u\cdot\nabla(V^j_{\eps}\phi)-\gamma_{\eps}\nabla V^j_{\eps}\cdot \nabla (U\phi)\, dx
-\alpha\int_{\om}(\nabla V^j\cdot\nabla(u_{\eps}\phi)-\gamma_{\eps}\nabla u_{\eps}\cdot \nabla (V^j\phi))\, dx\\
&\quad{}-\alpha\int_{\om}(\nabla u\cdot V^j_{\eps}\nabla\phi-\gamma_{\eps}\nabla V^j_{\eps}\cdot u\nabla\phi)\, dx
+\alpha\int_{\om}(\nabla V^j\cdot u_{\eps}\nabla\phi-\gamma_{\eps}\nabla u_{\eps}\cdot V^j\nabla \phi)\, dx\\
&=\alpha \int_{\om}(\gaep\nabla\uep\cdot\nabla(V^j_{\eps}\phi)
-\nabla V^j\cdot\nabla (u\phi))\, dx
+\int_{\om}(\uep-u) V^j_{\eps}\phi\,dx\\
&\quad{}-\alpha\int_{\om}(\gaep\nabla V_{\eps}^j\cdot\nabla(\uep\phi)-\nabla u\cdot\nabla(V^j\phi))\,dx-\int_{\om}(\uep-u)V^j\phi\,dx\\
&\quad{}-\alpha\int_{\om}(\nabla u\cdot V^j_{\eps}\nabla\phi-\gamma_{\eps}\nabla V^j_{\eps}\cdot u\nabla\phi)\, dx
+\alpha\int_{\om}(\nabla V^j\cdot u_{\eps}\nabla\phi-\gamma_{\eps}\nabla u_{\eps}\cdot V^j\nabla \phi)\, dx\\
&=\int_{\om}(\uep-u)(V^j_{\eps}-V^j)\phi\, dx+\alpha\int_{\om}
(\gaep\nabla\uep\cdot V^j_{\eps}\nabla\phi-\nabla V^j\cdot u\nabla\phi)\, dx\\
&\quad{}-\alpha\int_{\om}(\gaep\nabla V^j_{\eps}\cdot\uep\nabla\phi-\nabla u V^j\cdot \nabla\phi)\, dx
-\alpha \int_{\om}(\nabla u\cdot V^j_{\eps}\nabla\phi-\gaep\nabla V_{\eps}^j\cdot u\nabla\phi)\, dx\\
&\quad{}+\alpha\int_{\om}(\nabla V^j\cdot \uep\nabla\phi-\gaep\nabla \uep\cdot V^j\nabla\phi)\, dx\\
&=\int_{\om}(\uep-u)(V^j_{\eps}-V^j)\phi\, dx
+\alpha\int_{\om}\nabla V^j\cdot(\uep-u)\nabla\phi\, dx
-\alpha\int_{\om}\nabla u\cdot(V^j_{\eps}-V^j)\nabla\phi\, dx\\
&\quad{}-\alpha\int_{\om}\gaep\nabla V^j_{\eps}\cdot(\uep-u)\nabla\phi\, dx
+\alpha\int_{\om}\gaep\nabla \uep\cdot(V^j_{\eps}-V^j)\nabla\phi\, dx\\
&=\int_{\om}(\uep-u)(V^j_{\eps}-V^j)\phi\, dx+\alpha\int_{\om}\gaep\nabla (u_{\eps}-u)\cdot (V^j_{\eps}-V^j)\nabla\phi \, dx \\[-0.1cm]
&\quad{}-\alpha\int_{\om}\gaep\nabla (V^j_{\eps}-V^j)\cdot (u_{\eps}-u)\nabla\phi \, dx
+\alpha\int_{\omep}(\kappa-1)\nabla u\cdot(V^j_{\eps}-V^j)\nabla\phi \, dx\\
&\quad{} -\alpha\int_{\omep}(\kappa-1)\nabla V^j\cdot(\uep-u) \nabla\phi \, dx\,.
\end{align*}
Now, by Lemma \ref{EnEst},  (\ref{energy3}),  (\ref{energy4}), Schwarz inequality and using finally the regularity of $u$ and of $V^j$ we get (\ref{nuM}) and the claim follows.
\end{proof} 
We now state several properties of the polarization tensor $\mathcal{M}$ established in \cite{CapVog03a,CapVog06} that we will use in the sequel. 
From the definition of the tensor $\mathcal{M}$  given in \cite{BerCapGouFra09}, it is easy to see that it is symmetric and satisfies
\begin{equation}\label{bound1}
|\xi|^2\leq \mathcal{M}\xi\cdot\xi\leq \frac{1}{\kappa}|\xi|^2
\end{equation}
for any $\xi\in\mathbb{R}^2$.
Moreover
\begin{equation}\label{bound2}
\textrm{ tr }\mathcal{M}\leq 1+\frac{1}{\kappa}\,,
\end{equation}
\begin{equation}\label{bound3}
\textrm{ tr }\mathcal{M}^{-1}\leq 1+\kappa\,.
\end{equation}
Furthermore, in the case of constant coefficients by insertion of $\phi=\xi_i\xi_j$ in (\ref{defM1}) we get
\begin{equation}\label{bound4}
\mathcal{M}_{ij}\xi_i\xi_j=|\omep|^{-1}\int_{\omep}\nabla V_{\eps}\cdot\xi\, dx+o(1)=|\omep|^{-1}\int_{\omep}\nabla V_{\eps}\cdot\nabla V\, dx+o(1)\,,
\end{equation}
where $V_{\eps}=V_{\eps}^i\xi_i$ and $V=V^i\xi_i$. Hence, we can write
\begin{equation}\label{bound5}
\mathcal{M}\xi\cdot\xi=|\xi|^2+|\omep|^{-1}\int_{\omep}\nabla W_{\eps}\cdot\xi\, dx+o(1)\,,
\end{equation}
where $W_{\eps}=V_{\eps}-V$ is the solution to
\begin{equation}
\label{corr}
\left\{ \begin{aligned}
  \divg (\gaep \nabla W_\eps) &= \divg ((1-\kappa)1_{\omep}\xi) &&\text{in }\Omega\,,\\[0.1cm]
  \gaep \dfrac{\partial W_{\eps}}{\partial \nu} &= 0 && \text{on }\partial\Omega\,.
\end{aligned}\right.
\end{equation}

We are now ready to prove the following result:
\begin{proposition}\label{eigenvalue}
 We have
 \begin{equation}\label{eig}
   \mathcal{M} \tau\cdot \tau = 1,
   \qquad
   \mathcal{M} n \cdot n = \frac{1}{\kappa}.
 \end{equation}
\end{proposition}
\begin{proof}[Proof of Proposition \ref{eigenvalue}]
  Without loss of generality we may assume that $\tau = e_1 = (1,0)$
  and $n = e_2 = (0,1)$ are the standard basis vectors in $\R^2$.
  
  Let us set $\xi = \tau = e_1$ and denote by $W_{\eps}^1$ the corresponding solution of
  (\ref{corr}). We will first show that
\begin{equation}\label{eig1}
|\omep|^{-1}\int_{\omep}\nabla W^1_{\eps}\cdot e_1\, dx=o(1).
\end{equation}
Let $\omep'=\{ x+\rho e_2: x\in\sigma_{\eps}, \, -\eps^2\leq\rho\leq\eps^2\}$ and let us write
\[
\int_{\omep}\nabla W^1_{\eps}\cdot e_1\, dx=\int_{\omep'}\nabla W^1_{\eps}\cdot e_1\, dx+\int_{\omep\backslash\omep'}\nabla W^1_{\eps}\cdot e_1\, dx:=I_1+I_2\,.
\]
Observe that
\[
|I_2|\leq \|\nabla W^1_{\eps}\|_{L^2(\om)}|\omep\backslash\omep'|^{1/2}\,,
\]
and by the energy estimates 
\[
|I_2|\leq |\omep|^{1/2}|\omep\backslash\omep'|^{1/2}=o(|\omep|)\,.
\]
Let us now estimate $I_1$
\[
|I_1|=\left|\int_{-\eps^2}^{\eps^2}\int_{-\eps}^{\eps}\frac{\partial W^1_{\eps}}{\partial x_1}\, dx_1 dx_2\right|=\left|\int_{-\eps^2}^{\eps^2}W^1_{\eps}|_{-\eps}^{\eps}\, dx_2\right|\,.
\]
Observe now that by standard regularity results
\[
\|W^1_{\eps}\|_{L^{\infty}(\omep)}\leq C(\|W^1_{\eps}\|_{H^1(\om)}+\|1_{\omep}e_1\|_{L^q(\omep)})
\]
for $q>2$. Hence
\[
\|W^1_{\eps}\|_{L^{\infty}(\omep)}\leq C |\omep|^{1/q}
\]
and if $q\in(2, 3)$ we get
\[
|I_1|=o(|\omep|).
\]
Summarizing
\[
\left|\int_{\omep}\nabla W^1_{\eps}\cdot e_1\, dx\right|\leq |I_1|+|I_2|=o(|\omep|)\,,
\]
which proves (\ref{eig1}). Finally, inserting (\ref{eig1}) in (\ref{bound5}) and letting $\eps\rightarrow 0$, 
\[
\mathcal{M}e_1\cdot e_1=1.
\]
Recalling (\ref{bound2}) and (\ref{bound3}) we get that 
\[
\mathcal{M}e_2\cdot e_2=\frac{1}{\kappa}. 
\]
\end{proof}
We are now ready to prove our main result:
\begin{proof}[Proof of Theorem \ref{thm:Gasymptotic}]
Let $\omep'$ be defined as in the proof of Proposition \ref{eigenvalue}. Then, 
by Lemma \ref{Loc}, we can write
\[
\begin{aligned}
\mathcal{J}(\uep, v_{\eps})-\mathcal{J}(u,v)
&=\frac{\alpha(\kappa-1)}{2}\int_{\omep}\nabla \uep\cdot\nabla u\, dx\\
&=\frac{\alpha(\kappa-1)}{2}\int_{\omep'}\nabla \uep\cdot\nabla u\, dx+\frac{\alpha(\kappa-1)}{2}\int_{\omep\backslash\omep'}\nabla \uep\cdot\nabla u\, dx\,.
\end{aligned}
\]
Observe now that
\[
\int_{\omep\backslash\omep'}\nabla \uep\cdot\nabla u\, dx=\int_{\omep\backslash\omep'}\nabla (\uep-u)\cdot\nabla u\, dx+\int_{\omep\backslash\omep'}|\nabla u|^2\, dx\,.
\]
Using Schwarz inequality, the regularity estimates of Lemma \ref{Reg} and Lemma \ref{EnEst} we get
\[
\left|\int_{\omep\backslash\omep'}\nabla \uep\cdot\nabla u\, dx\right|
\leq C(\|\uep-u\|_{H^1(\om)}|\omep\backslash\omep'|^{1/2}+ |\omep\backslash\omep'|)\leq C|\omep\backslash\omep'|=o(|\omep|)\,.
\]
Hence 
\[
\mathcal{J}(\uep, v_{\eps})-\mathcal{J}(u,v)=\frac{\alpha(\kappa-1)}{2}\int_{\omep}\nabla \uep\cdot\nabla u\, dx=\frac{\alpha(\kappa-1)}{2}\int_{\omep'}\nabla \uep\cdot\nabla u\, dx+o(|\omep|).
\]
Let us choose
some vector function $\Phi\in C_0^0(\Omega; \mathbb{R}^2)$ such that 
\begin{equation}
\Phi(x)=\begin{cases}
\nabla u& x\in L_0,\\
0& x\in\Omega\setminus \hat L_0.\end{cases}
\end{equation}
Then, from Lemma \ref{nu} we get
\[
|\omep'|^{-1}\int_{\omep'}\nabla u_{\eps}\cdot\nabla u\,dx\rightarrow
\mathcal{M}\nabla u(y)\cdot\nabla u(y)
\]
as $\eps\rightarrow 0$.
Observing that $|\omep'|=4\eps^3$ we get
\[
\mathcal{J}(\uep, v_{\eps})-\mathcal{J}(u,v)=\frac{\alpha(\kappa-1)}{2}4\eps^3\mathcal{M}\nabla u (y)\cdot\nabla u(y)+o(\eps^3)\,.
\]
Finally, observing that the remainder term is uniformly bounded  with respect to $y\in L_0$, i.e., $|o(\eps^3)|\leq C\eps^{3+\eta}$ for some $\eta>0$ and $C$ and $\eta$ depend only on $\kappa, \delta_0, \|f\|_{H^{-1}(\om)}, \|f\|_{L^{\infty}(\om)}$ and $\nabla u$  is continuous on the compact set $L_0$ the claim follows.
\end{proof}

\section{Numerical Implementation}\label{sec_numerics}

We now propose two variants of an algorithm to edges detection that is based on the
expansion derived above, which states that
\begin{equation}\label{eq:approx_num}
{\mathcal J}_{\eps}(u_{\eps},v_{K \cup \omega_\eps}) - {\mathcal J}_{\eps}(u,v_K) 
\approx 2\beta\eps - 2\alpha(1-\kappa)\eps^3\mathcal{M}\nabla u(y)\cdot\nabla u(y)
\end{equation}
with $\mathcal{M} =\frac{1}{\kappa}n\otimes n+\tau\otimes\tau$.
For fixed $y \in \Omega$, the right hand side in~\eqref{eq:approx_num} is minimal 
for $\tau$ equal to the unit vector perpendicular to $\nabla u(y)$
in which case
\[
\mathcal{M}\nabla u(y)\cdot\nabla u(y) = \frac{1}{\kappa} \abs{\nabla u(y)}^2
\]
and
\[
{\mathcal J}_{\eps}(u_{\eps},v_{K \cup \omega_\eps}) - {\mathcal J}_{\eps}(u,v_K) 
\approx 2\beta\eps - 2\alpha\eps^3\frac{1-\kappa}{\kappa}\abs{\nabla u(y)}^2.
\]
As a consequence, we can expect a decrease of the function $\mathcal{J}_{\eps}$
in case
\[
\abs{\nabla u(y)}^2 \ge \frac{\beta\kappa}{\alpha\eps^2(1-\kappa)}
\]
and the decrease is maximal at points $y$ where the gradient of $u$
is maximal.

\begin{enumerate}
\item
Our first algorithm computes a smoothed version $u_s$ of the
input image $f$ a-priori, and then finds, using only the smoothed
image $u_s$, a sequence of edge indicators $K^{(k)}$,
where $K^{(k+1)}$ is formed from $K^{(k)}$ by the addition
of a strip $\sigma_\eps(x^{(k)},\tau^{(k)})$ for
which the expected decrease in the approximated functional 
${\mathcal J}_{\eps}$ from \eqref{Jeps} is maximal.
Theorem~\ref{thm:Gasymptotic} indicates that, as long as we only add strips
that are away from $K^{(k)}$, this is the case if
$x^{(k)}$ is chosen such that $\abs{\nabla u(x)}$ is maximal
and $\tau^{(k)} = (\nabla u(x^{(k)}))^\perp$.
However, because the asymptotic expansion of Theorem~\ref{thm:Gasymptotic} is only valid
away from $K^{(k)}$, we have to restrict the search for a maximum
of $\abs{\nabla u}$ to some set $L^{(k)}$ which is compactly contained
in $\Omega\setminus K^{(k)}$.
For instance, one can set $L^{(k)} := \Omega \setminus (\partial\Omega\cup K^{(k)}+B_\delta)$
for some $\delta > 0$;
a different construction, which we have used in the numerical
examples, is described below.
The iteration is stopped when the expected decrease of the 
gradient term in the functional is compensated by the increase in the edge term.
This is the case when $\abs{\nabla u(x^{(k)})}^2 <
\frac{\beta\kappa}{\alpha\eps^2(1-\kappa)}$.
This method is summarized in Algorithm~\ref{alg_no_update}.

\begin{algorithm}[ht]
  \SetLine
  \KwData{input image $f \colon \Omega \to \R$, parameters $\alpha$, $\beta > 0$,
    $\eps > 0$, $0 < \kappa < 1$\;}
  \KwResult{edge indicator set $K$\;}
  \BlankLine
  \SetKwInput{Initialize}{Initialization}
  \Initialize{set $K = \emptyset$ and $L = \Omega\setminus(\partial\Omega+B_\delta)$}
  \BlankLine
  compute the solution $u$ of
  \[
  \begin{cases}
    u-\alpha\divg(\nabla u) = f &\text{in }\Omega,\\
    \partial_\nu u = 0 & \text{on } \partial\Omega.
  \end{cases}
  \]
  \Repeat{$\abs{\nabla u(x^*)}^2 < \frac{\beta\kappa}{\alpha\eps^2(1-\kappa)}$}{
    find $x^* \in L$ with $\abs{\nabla u(x^*)}$ maximal\;
    compute a strip $\sigma_\eps$ of size $\eps$ centered at $x^*$ with normal $\nabla u(x^*)$\;
    set $K \leftarrow K \cup \sigma_\eps$\;
    compute an enlargement $S$ of $\sigma_\eps$\;
    set $L \leftarrow L \setminus S$\;
  }
  \caption{Implementation without updates of the smoothed
  function}\label{alg_no_update}
\end{algorithm}
In fact this algorithm is an anisotropic edge detector, which take
into account edge magnitudes and local edge orientations. 

\item In our second algorithm, we combine updates of the edge
  indicator with updates of the smoothed function $u$: After adding a
  fixed number $n_{\max}$ of strips to the edge set $K$, we define 
  the new diffusivity $v$ by
  \[
  v(x) :=
  \begin{cases}
    \kappa &\text{if } x \in K,\\
    1 &\text{if } x \not\in K,
  \end{cases}
  \]
  and then compute a corresponding smoothed function $u$, which is
  then used for selecting the next at most $n_{\max}$ strips in the
  edge set.
  The process of alternating between the addition of strips and
  updates of the smoothed function $u$ is repeated until no more
  admissible points $x \in L$ exist for which $\abs{\nabla u(x)}^2 \ge
  \frac{\beta\kappa}{\alpha\eps^2(1-\kappa)}$.
  
  The rationale behind this idea is the fact that the expansion
  derived above, though still valid, becomes 
  increasingly inaccurate as the number of added strips becomes larger.
  Therefore, at some point some reinitialization is necessary.
  Note, however, that the number $n_{\max}$ of strips that are added
  in each iteration mainly determines the computation time, as the
  computation of $u$ is the most costly part of the algorithm. Thus
  the number $n_{\max}$ should not be chosen too small. In the
  numerical implementations, we chose $n_{\max}$ in such a way that
  approximately 10 computations of $u$ were needed.

  The resulting method is described in Algorithm~\ref{alg_update}.

  \begin{algorithm}[ht]
    \SetLine
    \KwData{input image $f \colon \Omega \to \R$, parameters $\alpha$, $\beta > 0$,
      $\eps > 0$, $0 < \kappa < 1$, $n_{\max} \in \N$\;}
    \KwResult{edge indicator function $v$, smoothed image $u$\;}
    \BlankLine
    \SetKwInput{Initialize}{Initialization}
    \Initialize{set $v(x) = 1$ for $x \in \Omega$,
      $K = \emptyset$, and $L = \Omega\setminus(\partial\Omega+B_\delta)$\;}
    \BlankLine
    compute the solution $u$ of
    \[
    \begin{cases}
      u-\alpha\divg(\nabla u) = f &\text{in }\Omega,\\
      \partial_\nu u = 0 & \text{on } \partial\Omega;
    \end{cases}
    \]
    \Repeat{$\max_{x^*\in L}\abs{\nabla u(x^*)}^2 < \frac{\beta\kappa}{\alpha\eps^2(1-\kappa)}$}{
      set $n = 1$\;
      \Repeat{$n > n_{\max}$ or $\abs{\nabla u(x^*)}^2 <
        \frac{\beta\kappa}{\alpha\eps^2(1-\kappa)}$}{
        find $x^* \in L$ with $\abs{\nabla u(x^*)}$ maximal\;
        compute a strip $\sigma_\eps$ of size $\eps$ centered at $x^*$ with normal $\nabla u(x^*)$\;
        set $K \leftarrow K \cup \sigma_\eps$\;
        compute an enlargement $S$ of $\sigma_\eps$\;
        set $L \leftarrow L \setminus S$\;
        set $n \leftarrow n+1$\;
      }
      set $v(x) = \kappa$ for $x \in K$\;
      compute the solution $u$ of
      \[
      \begin{cases}
        u-\alpha\divg(v\nabla u) = f &\text{in }\Omega,\\
        \partial_\nu u = 0 & \text{on } \partial\Omega;
      \end{cases}
      \]
    }
    \caption{Implementation with updates of $u$}\label{alg_update}
  \end{algorithm}
\end{enumerate}

\subsection*{Solution of the PDE}

For the numerical solution of the equation
\[
\begin{cases}
  u-\alpha\divg(v\nabla u) = f &\text{in }\Omega,\\
  \partial_\nu u = 0 & \text{on } \partial\Omega,
\end{cases}
\]
we have implemented a finite element method using bilinear ansatz
functions on a rectangular grid for $u$ and piecewise constant
ansatz functions on the same grid for the diffusivity $v$.
The solution of the resulting linear equation was computed with the CG
method.

\subsection*{Update of the Edge Indicator}

For updating the edge indicator set $K$ (and the function $v$),
we have to find maximizers of $\abs{\nabla u}$.
We restrict the search to midpoints of the rectangular elements $E_k$ of the finite elements and evaluate 
the gradient on the elements analytically.

\begin{figure}
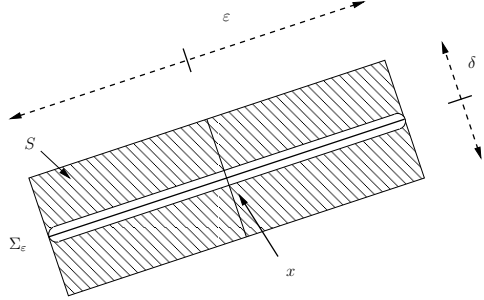

\[
\resizebox{0.5\textwidth}{!}{
\input ./Sigma_eps.tex}
\]
\caption{\label{fig:S}
  Sketch of the construction of the set $\sigma_\eps$
  and the corresponding enlarged set $S$}
\end{figure}

Assume now that the maximum of $\abs{\nabla u}$ is attained at $y^*$.
For the update of the set $L$ we define the enlargement $S$ of
$\Sigma_\eps$ as a rectangle of side-lengths $2\eps$ and $2\delta$ for
some $0 < \delta < \eps$ around the center-line of the strip. That is,
(see Figure~\ref{fig:S})
\[
S := \set{y : \dist(y_k,\sigma_\eps) \le \delta \text{ and }
  \abs{(y-y_k)\cdot \nabla u(y_k)^\perp}\le \eps\abs{\nabla u(y_k)}}
\]
for some $\delta > 0$.
In the numerical experiments below we have chosen $\eps = 3h$ and
$\delta = 2h$ with $h$ being the pixel distance.

\subsection*{Numerical Experiments}

We have tested the two algorithms proposed above using the
\emph{Parrots} image (see Figure~\ref{fi:parrot}, upper left). In
addition, we provide a comparison with the results of the algorithm
proposed in~\cite{GraMusSch11_report}, where balls instead of strips
are used for covering the edge set.

\begin{figure}
  \[
  \begin{aligned}
    &\includegraphics[width=0.45\textwidth]{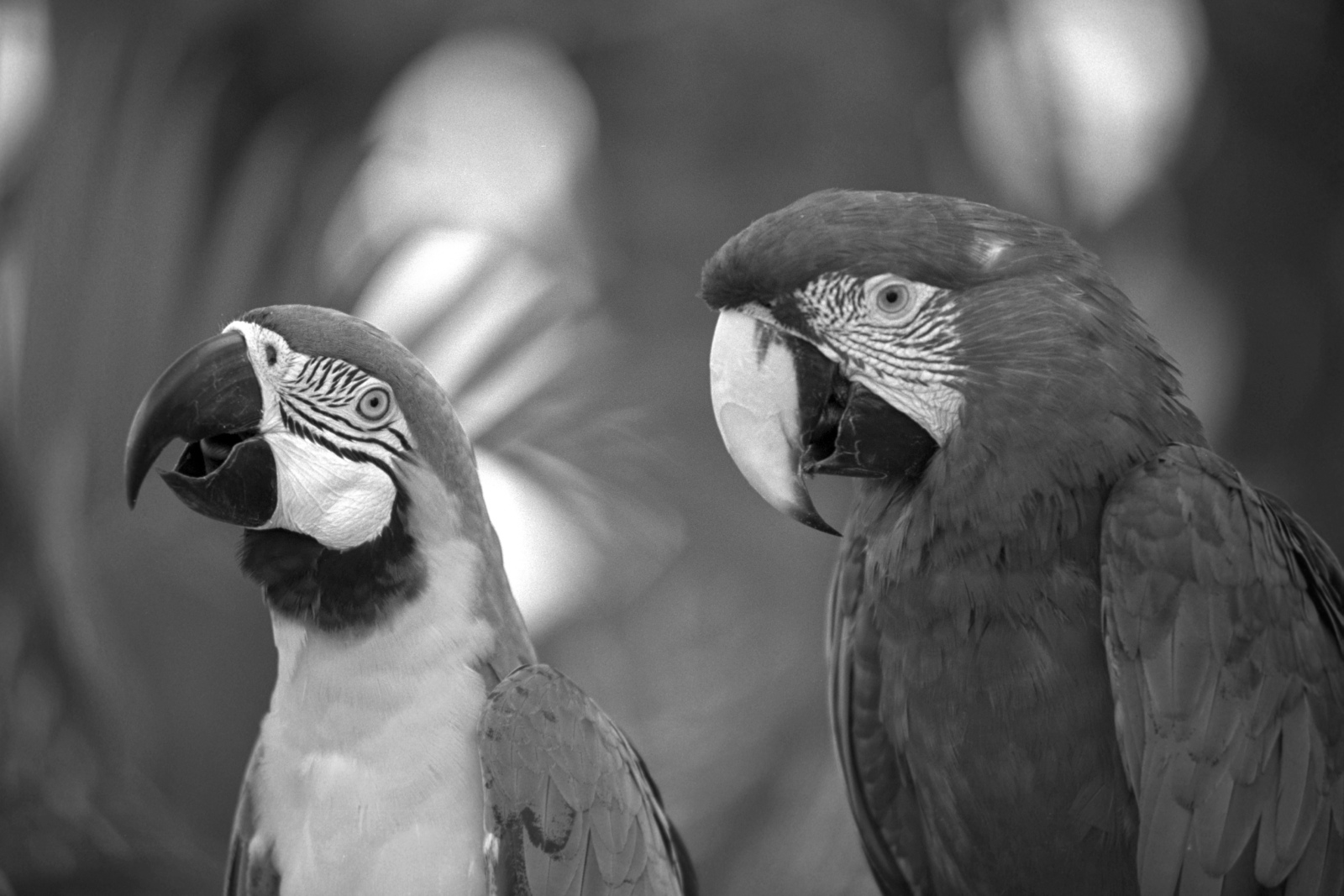}
    &\includegraphics[width=0.45\textwidth]{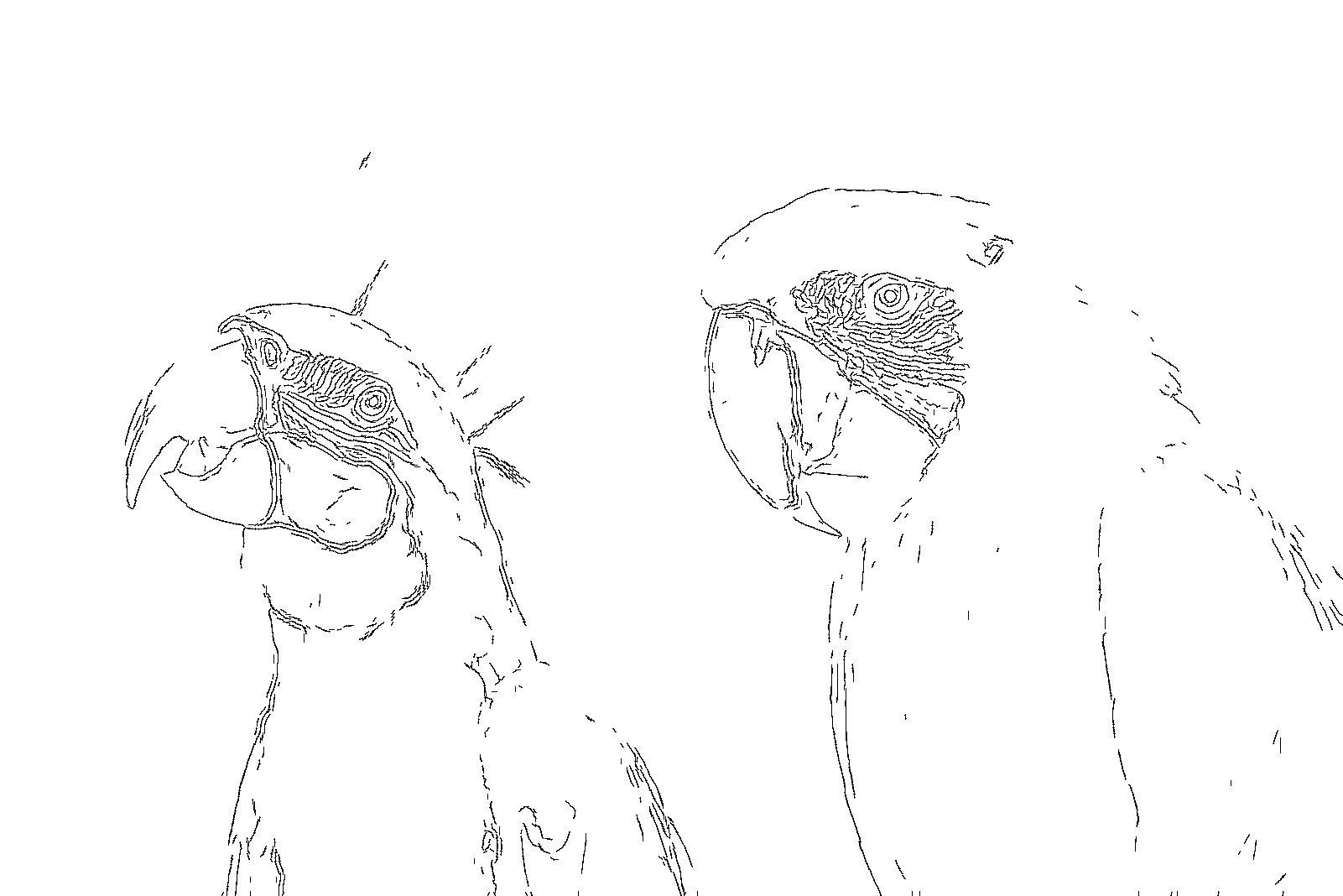}\\
    &\includegraphics[width=0.45\textwidth]{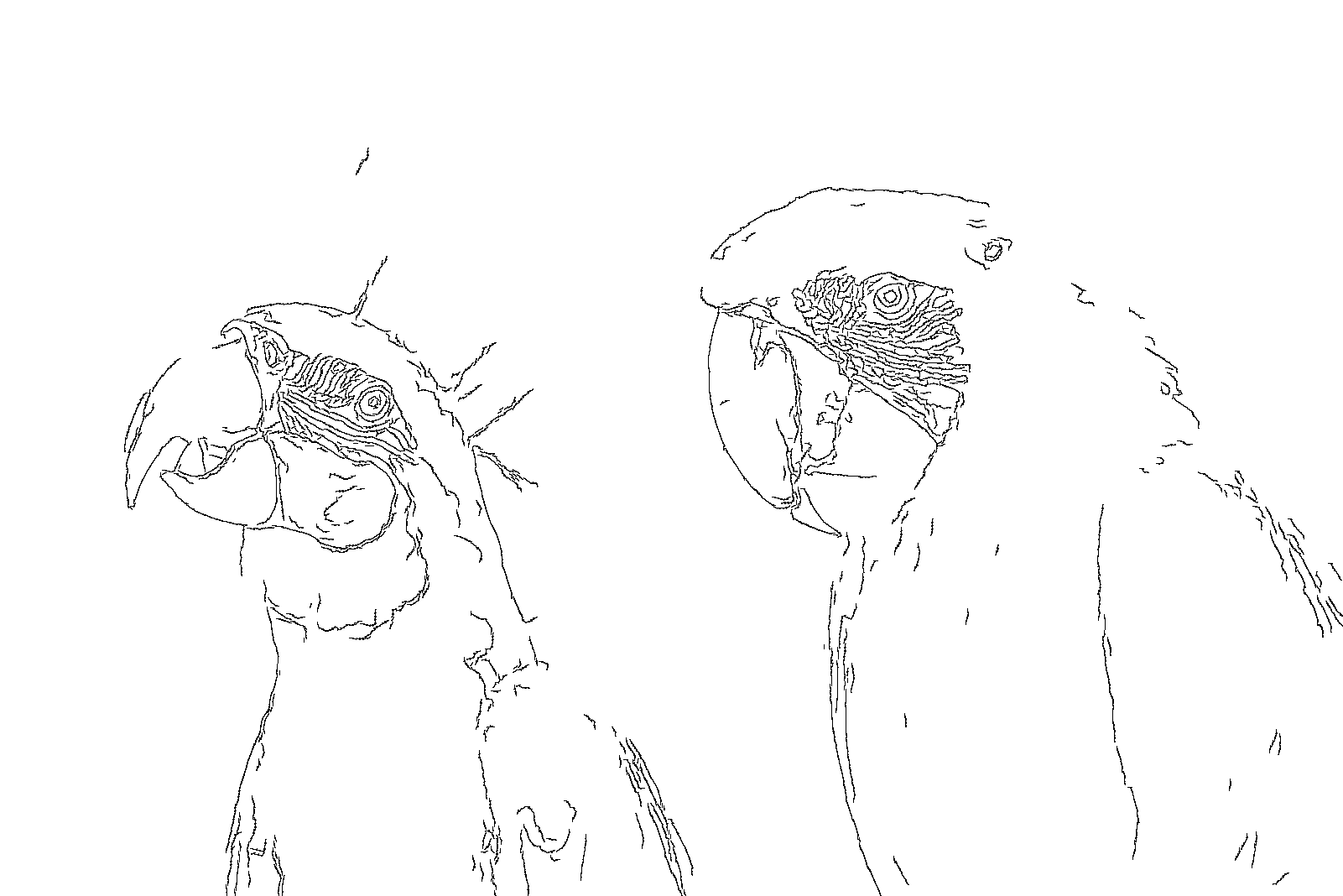}
    &\includegraphics[width=0.45\textwidth]{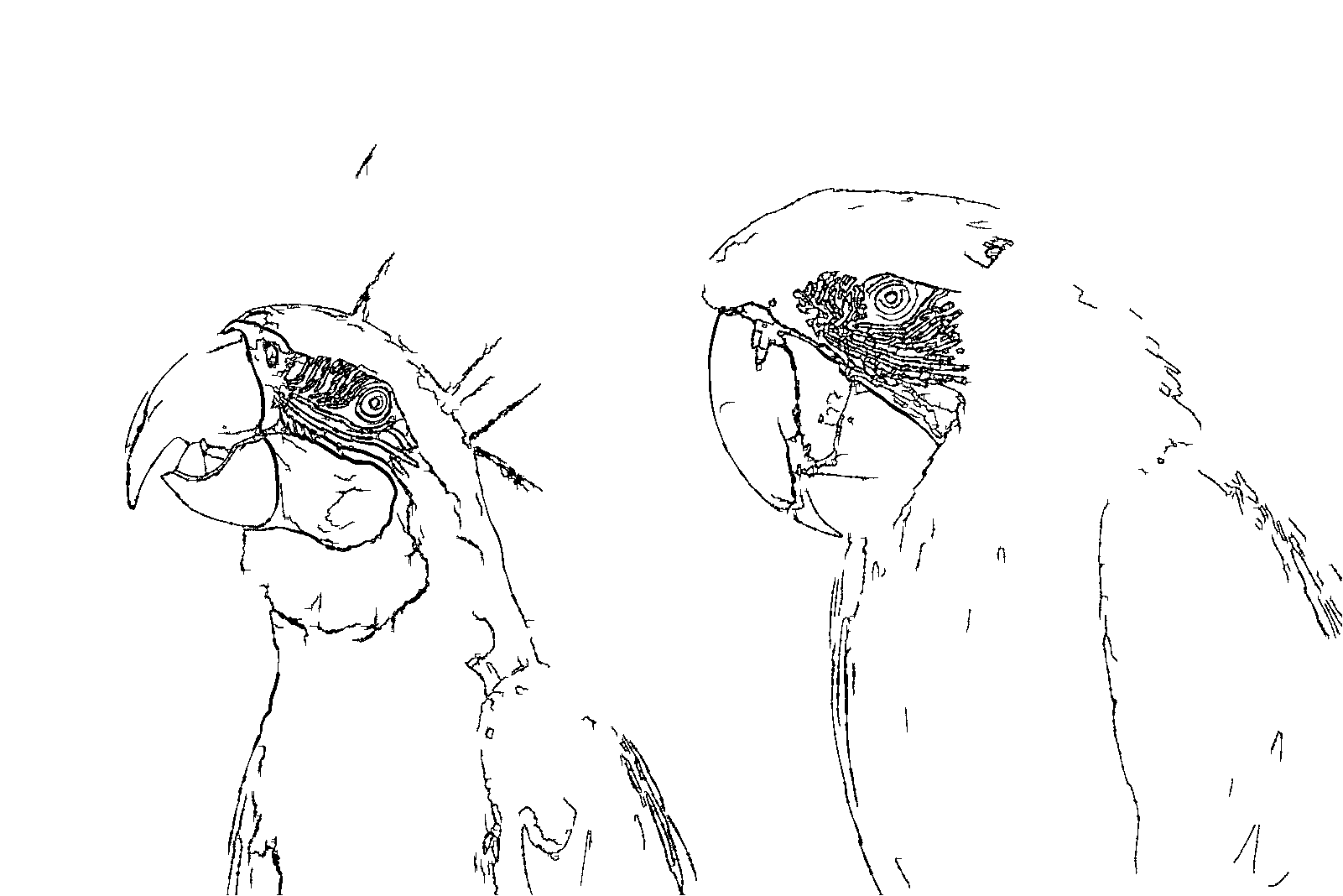}\\
  \end{aligned}
  \]
  \caption{\label{fi:parrot}
    \emph{Upper left:} Original Image. \emph{Upper right:}
    Result with Algorithm~\ref{alg_no_update}.
    \emph{Lower left:} Result with Algorithm~\ref{alg_update}.
    \emph{Lower right:} Result using the Algorithm
    from~\cite{GraMusSch11_report}.
    In all examples, the parameters were $\alpha = 8$ and $\beta = 150$.
    }
\end{figure}

\begin{figure}
  \[
  \begin{aligned}
    &\includegraphics[width=0.3\textwidth]{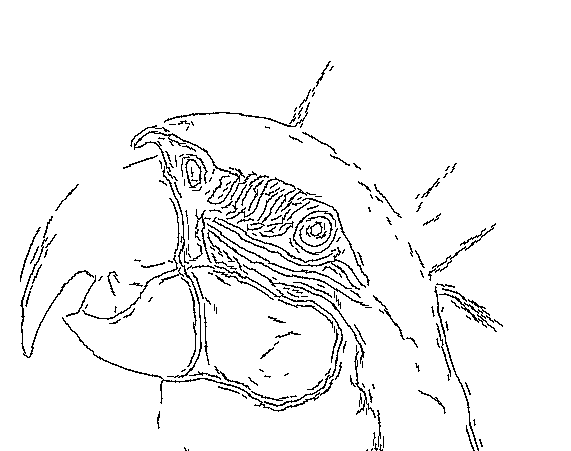}&
    &\includegraphics[width=0.3\textwidth]{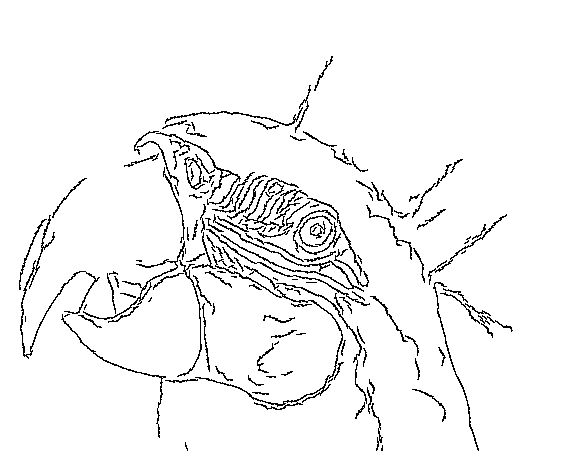}&
    &\includegraphics[width=0.3\textwidth]{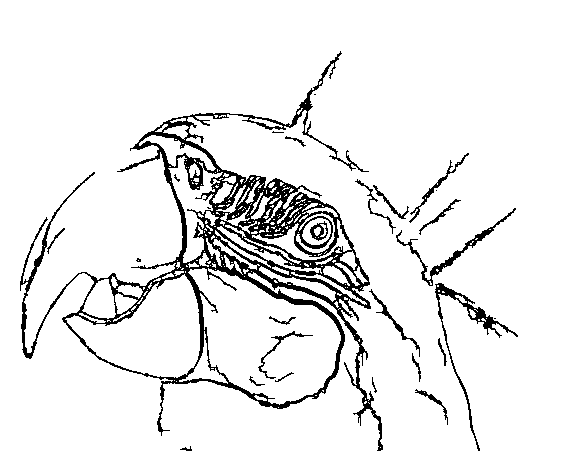}
  \end{aligned}
  \]
  \caption{\label{fi:parrot_detail}
    Close up view on a detail of Figure~\ref{fi:parrot}.
    \emph{Left:} Result with Algorithm~\ref{alg_no_update}.
    \emph{Middle:} Result with Algorithm~\ref{alg_update}.
    \emph{Right:} Result with the method
    from~\cite{GraMusSch11_report}.
    Note in particular the thick edges in the last image.
  }
\end{figure}

\begin{figure}
  \[
  \begin{aligned}
    &\includegraphics[width=0.2\textwidth]{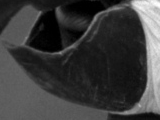}&
    &\includegraphics[width=0.2\textwidth]{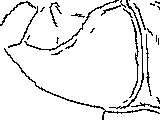}&
    &\includegraphics[width=0.2\textwidth]{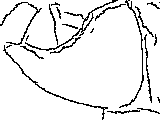}&
    &\includegraphics[width=0.2\textwidth]{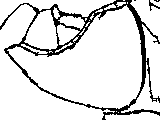}
  \end{aligned}
  \]
  \caption{\label{fi:parrot_beak}
    Close up view of the lower part of the beak of the first parrot in
    Figure~\ref{fi:parrot}.
    \emph{First:} Original image.
    \emph{Second:} Result with Algorithm~\ref{alg_no_update}.
    \emph{Third:} Result with Algorithm~\ref{alg_update}.
    \emph{Fourth:} Result with the method
    from~\cite{GraMusSch11_report}.
    One clearly sees the spurious edges in the first segmentation and 
    the thick edge in the third segmentation, which is partly resolved
    in the second one.
  }
\end{figure}

While, generally speaking, the positions of the detected edges do
approximately agree for the different algorithms, the actual form of
the edges may significantly differ. Thus the algorithm
of~\cite{GraMusSch11_report} results in thick edges, which do not
appear in the results from the strip based methods. This difference
is due to the fact that the directional information present in the
strips allows the exclusion of laterally neighboring pixels from
further considerations. In contrast, the balls that are used for edge
covering in~\cite{GraMusSch11_report} do not allow a similar exclusion
of pixels.

Concerning computation times, Algorithm~\ref{alg_no_update} is clearly
faster than Algorithm~\ref{alg_update}, as the main computational
effort of the methods lies in the solution of the PDE, which has to be
computed several times in the case of Algorithm~\ref{alg_update}. We
do note, however, that Algorithm~\ref{alg_no_update} introduces
artifacts in the form of parallel edges, which can be clearly seen in
the close up view of the parrot's head in
Figure~\ref{fi:parrot_detail} and~\ref{fi:parrot_beak}.
These false edge detections can be
attributed to the smearing out of edges occurring in the first
solution of the PDE.

\section*{Acknowledgments}
This work has been supported by the Austrian Science Fund (FWF) within the 
national research network Photoacoustic Imaging in Biology and Medicine, project S10505-N20, and Variational Methods on Manifolds, project S11704.
OS wants to thank Yves Capdeboscq for stimulating discussions. The work was initiated during a special semester at MSRI in 2010 and finished 
during a special semester at Mittag-Leffler in 2013; the hospitality of MSRI and Mittag-Leffler is gratefully acknowledged.

\section*{Appendix --- $\Gamma$-convergence}

In the following we show, similarly as in~\cite{GraMusSch11_report},
that the functional $\mathcal{J}_{\eps}$ defined in (\ref{Jeps})
$\Gamma$-converges as $\eps\to 0$ and $\kappa \to 0$ 
to the Mumford-Shah functional $\mathcal{F}\colon L^2(\Omega)\times L^2(\Omega) \to [0,+\infty]$, defined by
\begin{equation}
 \label{eq:ms2}
\mathcal{F}(u,v) = \begin{cases}
{\displaystyle
\frac{1}{2} \int_\Omega (u-f)^2\,dx 
+ \frac{\alpha}{2}\int_{\Omega\setminus S_u} \abs{\nabla u}^2\,dx + \beta\mathcal{H}^1(S_u)}
&\text{ if } v \equiv 1\,,\\
+ \infty &\text{ else.}
\end{cases}
\end{equation}
Here $S_u$ denotes the discontinuity set of the function $u$ (see \cite{Bra98}).

We do stress that, in contrast to the rest of the paper, where
$\kappa$ was constant, it is necessary for obtaining any non-trivial
$\Gamma$-convergence result that this parameter tends to zero much
faster than the size $\eps$ of the covering strips. Thus the following
theorem can be interpreted as saying that the minimizers of
$\mathcal{J}_\eps$ are close to minimizers of the Mumford--Shah
function, if both parameters $\eps$ and $\kappa$ are close to
zero. The asymptotic expansion derived in
Theorem~\ref{thm:Gasymptotic}, however, relies on $\kappa$ being
bounded away from zero; the constant in the $O(\eps^3)$ expansion
tends to $+\infty$ as $\kappa \to 0$.

\begin{theoremnn}
\label{th:gamma_convergence}
  Assume that $\kappa(\eps) = o(\eps^2)$ as $\eps \to 0$.
  Then, 
  \[
  \mathcal{F} = \Gammalim_{\eps\to 0}\mathcal{J}_{\eps}\;.
  \]
\end{theoremnn}
\begin{proof}
In order to prove the $\Gamma$-convergence result, we have to show
that
\[
\Gammalimsup_{\eps \to 0} \mathcal{J}_{\eps}
\le \mathcal{F}
\le \Gammaliminf_{\eps \to 0} \mathcal{J}_{\eps}.
\]

The proof of the $\liminf$-inequality is along the lines
of~\cite{GraMusSch11_report}. Therefore we only prove the
$\limsup$-inequality. 

  Following~\cite{GraMusSch11_report}, we introduce the set $\mathcal{W}(\Omega)$
  consisting of all functions $u \in \SBV(\Omega)$ for which the following hold:
  \begin{enumerate}
  \item $\mathcal{H}^{1}(\overline{S}_u\setminus S_u) = 0$.
  \item The set $\overline{S}_u$ is the union of a finite number of 
    almost disjoint line segments contained in $\Omega$, that is,
    their pairwise intersections are either empty or contain a single point.
  \item $u|_{\Omega\setminus \bar{S}_u} \in W^{1,\infty}(\Omega\setminus S_u)$.
  \end{enumerate}
  This set has been shown to be dense in $\SBV(\Omega)$ in the sense that,
  for every $u \in \SBV$, there exists a sequence $(u_j)_{j\in\N} \in \mathcal{W}(\Omega)$
  such that $\norm{u_j-u}_{L^2} \to 0$ and $\mathcal{F}(u_j) \to \mathcal{F}(u)$
  (see~\cite{Cor97,CorToa99}).

  Now assume that $u \in \mathcal{W}(\Omega)$ and $\eps > 0$.
  In the following, we will construct sequences $u^\eps \to u$ and $v^\eps
  \to 1$ such that
  \[
  \mathcal{J}_{\eps}(u^\eps,v^\eps) \to \mathcal{F}(u,1).
  \]
  Because of the aforementioned density of $\mathcal{W}(\Omega)$ and
  the fact that $\mathcal{F}(u,v) = +\infty$ for $v\neq 1$, this
  will prove the $\limsup$-part.

  Because $u \in \mathcal{W}(\Omega)$,
  there exists a finite number $k$ of almost disjoint line segments $[a_i,b_i] \subset \R^2$
  such that $S_u = \bigcup_{i=1}^k [a_i,b_i]$.
  Moreover,
  \[
  \mathcal{H}^1(S_u) = \mathcal{H}^1(\overline{S}_u) = \sum_{i=1}^k \abs{b_i-a_i}\;.
  \]
  Now choose a minimal number of points $y_j^{(i)} \in [a_i,b_i]$, $j = 1,\ldots,l_i$, 
  in such a way that the union of the strips
  $\omega_\eps\bigl(y_j^{(i)},\frac{b_i-a_i}{\abs{b_i-a_i}}\bigr)$ covers the set
  \begin{equation*}
    K_i^{\eps} := \{x \in \R^2 \,:\, \dist(x,[a_i,b_i]) < \eps^2\}\;.
  \end{equation*}
  This can be achieved with at most $1+\frac{\abs{b_i-a_i}}{2\eps}$ points.
  Define 
  \[
  S_\eps := \bigcup_{i=1}^k \bigcup_{j=1}^{l_i}
  \omega_\eps\bigl(y_j^{(i)},\frac{b_i-a_i}{\abs{b_i-a_i}}\bigr)
  \]
  and let $v^{\eps} := v_{S_\eps}$.
  Noting that
  \[
  \mathcal{L}^2(S_\eps)\le \eps^2\bigl((2+2k\eps)\mathcal{H}^1(S_u)+k\pi\eps^2\bigr),
  \]
  we see that $v^\eps \to 1$ as $\eps \to 0$.
  Moreover 
  \[
  m_{\eps}(v^{\eps}) \le \sum_{i=1}^k\Bigl(1+\frac{\abs{b_i-a_i}}{2\eps}\Bigr)
  \le k+\frac{\mathcal{H}^1(S_u)}{2\eps}\,,
  \]
  showing that
  \[
  \limsup_{\eps\to 0} 2\beta \eps m_{\eps}(v^\eps) \le \beta\mathcal{H}^1(S_u)\;.
  \]
  Define moreover 
  \[
  u^{\eps}(x) := u(x) \min\Bigl(\frac{\dist(x,S_u)}{\eps^2},1\Bigr)\;.
  \]
  Then Lebesgue's theorem of dominated convergence implies that
  $u^{\eps} \to u$ in $L^2(\Omega)$, and therefore
  \[
  \int_\Omega (u^\eps-f)^2\,dx \to \int_\Omega (u-f)^2\,dx \text{ as } \eps \to 0\;.
  \]
  Moreover, $\nabla u^{\eps}(x) = \nabla u(x)$ for $x \not\in K^{\eps}:=\{x\in\R^2 \,:\, \dist(x,S_u) < \eps^2\}$,
  and
  \[
  \abs{\nabla u^{\eps}(x)} \le \abs{\nabla u(x)}+\frac{\norm{u}_{L^\infty}}{\eps^2} \text{ for almost every } x \in K^{\eps}\,.
  \]
  This implies that
  \[
  \int_\Omega v^{\eps}\abs{\nabla u^{\eps}}^2\,dx
  \le \int_{\Omega\setminus K^{\eps}} \abs{\nabla u}^2\,dx 
  + 2\kappa(\eps)\int_{K^{\eps}\setminus S_u} \abs{\nabla u}^2+\frac{\norm{u}_{L^\infty}^2}{\eps^4}\,dx\;.
  \]
  Because
  \[
  \mathcal{L}^2(K^\eps)\le \eps^2\bigl(2\mathcal{H}^1(S_u)+k\pi\eps^2\bigr)
  \]
  and $\kappa(\eps) = o(\eps^2)$ as $\eps \to 0$, this shows that
  \[
  \limsup_{\eps \to 0}\int_\Omega v^{\eps}\abs{\nabla u^{\eps}}^2\,dx \le \int_{\Omega\setminus S_u}\abs{\nabla u^{\eps}}^2\,dx\;.
  \]
  Together, these estimates imply that 
  \[
  \limsup_{\eps \to 0} \mathcal{J}_{\eps}(u^\eps,v^\eps) \le \mathcal{F}(u,1)\,,
  \]
  which, because of the density of $\mathcal{W}(\Omega)$, in turn proves
  that
  \[
  \mathcal{F} \ge \Gammalimsup_{\eps\to 0} \mathcal{J}_{\eps}\;.
  \]
\end{proof}

\bibliographystyle{plain}

\def\cprime{$'$}

\end{document}